\documentclass[11pt, reqno, a4paper]{amsart}

\usepackage{amsmath}
\usepackage{amsfonts}
\usepackage{latexsym}
\usepackage{amssymb}
\usepackage{epsfig}
\usepackage[english]{babel}
\usepackage{graphicx}
\graphicspath{{figures/}}

\newcommand{\lra}{\longrightarrow}
\newcommand{\ie}{\leqslant}
\newcommand{\se}{\geqslant}
\newcommand{\bombscal}[3]{\left\langle #1\, \left| \, #2\right.\right\rangle_{(#3)} }

\newtheorem{theorem}{Theorem}
\newtheorem{proposition}{Proposition}
\newtheorem{cor}{Corollary}
\newtheorem{lemma}{Lemma}
\newtheorem{definition}{Definition}
\newtheorem{remark}{Remark}

\newtheorem*{claim}{Conjecture}

\begin{document}

\title[Sendov Conjecture]{Sendov conjecture for high degree polynomials}

\author{J\'er\^ome D\'egot}

\address{J\'er\^ome D\'egot \\ Lyc\'ee Louis-le-Grand \\
123 rue St Jacques\\
75 005 Paris} 

\email{jerome.degot@numericable.fr}



\date{\today}

\begin{abstract}
Sendov conjecture says that any complex polynomial $P$ having all its zeros
in the closed unit disk and $a$ being one of them, the closed disk of center $a$ and
radius $1$ contains a zero of the derivative $P'$. The main result of this paper is 
a proof of Sendov conjecture when the polynomial $P$ has a degree higher than a fixed 
integer $N$. We will
give estimates of this integer $N$ with respect to $|a|$. To obtain this result, we will
study the geometry of the zeros and critical points (i.e. zeros of $P'$) of a
polynomial which will eventualy contradict Sendov conjecture.   
\end{abstract}

\maketitle

\section{Introduction}
Sendov's conjecture can be state as follows.
\begin{claim}[]
	Let $P(z) = (z-a) \prod_{k=1}^{n-1} (z-z_k)$ be a monic complex polynomial 
	having all zeros in the closed unit disk there exists a zero $\zeta$ of 
	its derivative $P'$ such that :
	$|\zeta - a |\ie 1$.
\end{claim}
This conjecture appears for the first time in 1967 Hayman's book
{\it Research Problems in Function Theory} where it was improperly
attributed to the Bulgarian mathematician Illief. Since 1967 it was proved 
for a few particular cases, for example :
polynomials having at most $8$ distinct zeros  \cite{BX},
when $|a|=1$ \cite{Rub}, if $P$ vanishes at $0$, when zeros of $P$ and $P'$ are real
\cite{RS},
when all the summits of the convex hull of the zeros of $P$ lie on the unit
circle \cite{Sch}, but the general case is still open in spite of 80 papers
devoted to it. Surveys of the problem have been 
given by M. Marden \cite{Mar2} and Bl. Sendov \cite{Sen} and we refer the reader to these 
for further information and bibliographies. 

Fix $0<a<1$, in this paper we prove that there exists an integer $N$ such that
Sendov's conjecture is true for all polynomial of degree bigger than $N$.
Assuming that $P$ contradicts Sendov's conjecture,
we estimate below and above the
positive real number $|P(c)|$ for some $c$ satisfying $0<c<a$.
This leads to a contradiction for high values of the degree of $P$. 

\section{Background}

Denote by $\mathbb C_n[X]$ the $\mathbb C$-vector 
space of complex polynomials of degree less or equal to $n$.

\begin{definition}[Hermitian inner product]
	Consider $P$ and $Q\in\mathbb C_{n}[X]$ such that 
	$P(z) = a_{0} + a_{1} z + \cdots + a_{n} z^n$ and $Q(z) = b_{0} + b_{1} z + \cdots + 
	b_{n} z^n$, we define :
	\[
	\bombscal{P}{Q}{n} := \sum_{i=0}^n \binom{n}{i}^{-1} a_{i}\, 
	\overline{b_{i}}
	\]
\end{definition}

\begin{proposition}\label{rem:1}
For all $P\in\mathbb C_n[X]$ and $\alpha\in\mathbb C$ :
\[ 
P(\alpha) = \bombscal{P}{(\overline{\alpha} X+1)^{n}}{n}
\qquad\text{and}\qquad 
P'(\alpha) = n \bombscal{P}{X (\overline{\alpha}X+1)^{n-1}}{n} 
\]
\end{proposition}
\begin{proof}[Proof]
It is an immediate computation.
\end{proof}

\begin{theorem}[Walsh contraction principle \cite{Walsh}]
\label{thm:Wal}
Let $n\in\mathbb N$, $P\in\mathbb C_n[X]$ and $\alpha_1,...,\alpha_n\in\mathcal C$ 
where $\mathcal C$ is 
a circular region  of the complex projective plane
there exits $\beta\in\mathcal C$ such that :
\[
 \bombscal{P}{(\overline{\alpha_1}X+1)\cdots(\overline{\alpha_n}X+1)}{n}
	= \bombscal{P}{(\overline{\beta}X+1)^{n}}{n} 
\]
\end{theorem}

We deduce a result which seem not already known in spite of its simplicity.

\begin{theorem}[]\label{thm:3}
Let $P$ be a polynomial of $\mathbb C_n[X]$ and $\delta$ a complex number
satisfying $P'(\delta) \ne 0$. For any complex number $\omega$, the polynomial 
$P(z)-\omega$ has a zero in the disk whose diameter is the line segment
$\left[\delta, \delta - \frac{n(P(\delta) -\omega)}{P'(\delta)}\right]$.
\end{theorem}

\begin{proof}[Proof]
Write $R=\frac{P(\delta)-\omega}{P'(\delta)}$, 
using proposition \ref{rem:1} we obtain :
\begin{align*}
 & R P'(\delta) - P(\delta) + \omega = 0 \\
\iff & R \, n \bombscal{P}{X(\overline{\delta}X+1)^{n-1}}{n} -
\bombscal{P-\omega}{(\overline{\delta}X+1)^{n}}{n} = 0 \\
\iff & R \, n \bombscal{P - \omega}{X(\overline{\delta}X+1)^{n-1}}{n} -
\bombscal{P-\omega}{(\overline{\delta}X+1)^{n}}{n} = 0 \\
\iff & \bombscal{P - \omega}{(\overline{(\delta - n R)}X+1)
(\overline{\delta}X+1)^{n-1}}{n} = 0 
\end{align*}
Theorem \ref{thm:Wal} implies that there exits
$\lambda$ in the disk of diameter $[\delta,\delta - nR] $ satisfying :
\[
\bombscal{P-\omega}{(\overline\lambda X+1)^n}{n} = 0 \iff P(\lambda) - \omega = 0
\]
this proves the theorem.
\end{proof}

\begin{theorem}[Perpendicular bisector theorem]\label{media}
Let $P$ denote a polynomial, $\alpha$ and $\beta$ two complex numbers such that
$P(\alpha)=P(\beta)$ then the perpendicular bisector of the line segment 
$[\alpha,\beta]$ intersects the
convex hull of the zeros of $P'$ (i.e. each half-plane delimited by the bisector
contains at least one zero of $P'$). 
\end{theorem}
This result is a corollary of grace's theorem (see \cite{mard}).

\section{Notations}

From now on till the end of the paper $P$ denotes a monic polynomial which
contradicts Sendov's conjecture. Set the following notations :
\begin{align*}
P(z) & =  (z-a) \prod_{k=1}^{n-1} (z - z_k)  \\
P'(z) & =  n \prod_{k=1}^{n-1} (z - \zeta_k) 
\end{align*}
where $|a|\ie 1$, $|z_k| \ie 1$, 
$|\zeta_k|\ie 1$ (by Gauss-Lucas theorem) and $|a-\zeta_k| \se 1$ 
for $k=1,...,n-1$. Without loss of generality we may assume that $0<a<1$. 
We call $m$ the real part of the centroid of the zeros of $P$, it is a well known 
property that this centroid is invariant under
derivation thus :
\[
m=  \frac{1}{n}\Re\left(a+\sum_{k=1}^{n-1} z_k \right) 
=  \frac{1}{n-1}\Re\left(\sum_{k=1}^{n-1} \zeta_k \right)
\]
We define $p$ and $q$ by :
\[ p=\frac{a/2-m}{1-a/2}\in[0,1] \qquad\text{and}\qquad q=\frac{a/2-m}{1+a/2}\in[0,1]\]

\begin{theorem}[]\label{exclu}
Let $c$ denote a real number with $0<c<a$, $P$ has no zero in the disk 
of center $c$ and radius $1-\sqrt{1+c(c-a)}$.
\end{theorem}
\begin{proof}[Proof]
On the contrary, suppose that there exists a zero $\gamma$ of $P$ such that 
$|c-\gamma|\ie 1-\sqrt{1+c(c-a)}$, since $P(\gamma)=P(a)$ by perpendicular bisector 
theorem there exists a zero $\zeta$ of $P'$  such that
$|\gamma-\zeta|\se |a-\zeta|\se 1$ therefore :
\[
|c-\zeta| \se 1 - |\gamma -c| \se \sqrt{1+c(c-a)}
\] 
then
\[
c^2+1-2 c \Re(\zeta) \se c^2 +1 - a c \iff \Re(\zeta) \ie \frac{|\zeta^2|-1}{2c} + 
\frac{a}{2}
\]
We deduce
$\Re(\zeta) \ie a/2$ which is impossible 
because $|\zeta|\ie 1$ and $|a-\zeta|\se 1$, this proves the result.
\end{proof}

\section{Lemmas}

We give in this part some technical inequalities, for the convenience of the 
reader the proofs are differ to the end of the paper.
\begin{lemma}\label{lem1}
Let $\delta$ denote a real number satisfying $0<\delta<1$ and let $a_1$,...,$a_n$ denote 
complex numbers belonging to the closed unit disk, then
\[
\prod_{k=1}^n \left|\delta-a_k\right| \ie 
\left(\sqrt{1+\delta^2-2\delta s}\right)^n
\]
where $s=\frac{1}{n}\Re\left(\sum_{k=1}^n a_k\right)$.
\end{lemma}

\begin{lemma}\label{maj_deri}
Let $\delta$ denote a real number such that $0 < \delta < a <$ and let
$b_1,\ldots,b_{n-1}$ denote complex numbers satisfying $|b_k| \ie 1$ and 
$|b_k - a| \se 1$ for all $k\in\{1,...,n-1\}$, then :
\[
\prod_{k=1}^{n-1} \left| \frac{\delta - b_k}{a - b_k}\right| \ie A^{n-1} 
\]
where $A = \left(\frac{1+\delta}{1+a}\right)^{q} \sqrt{1 + \delta^2 - \delta a}^{1-q}$,
 $q=\frac{a/2-s}{1+a/2}$ and $s = \frac 1 {n-1}\Re\left(\sum_{k=1}^{n-1} 
b_k\right)$.
\end{lemma}
\begin{remark}
 Note that $A<1$ and $0\ie q\ie 1$. 
\end{remark}

\begin{lemma}\label{min_deri}
Under the assumptions of lemma \ref{maj_deri}. Let $b>1$, then :
\[
\prod_{k=1}^{n-1} \left| b - b_k\right| \se
\min(B_1,B_2)^{n-1} 
\] 
where $B_1 = \left(1+b-a\right)^{p} (\sqrt{1 + b^2 -  b a})^{1-p}$,
	$p=\frac{a/2-s}{1-a/2}$ and
	$B_2 = \left(1+b\right)^{q} (\sqrt{1 + b^2 - b a})^{1-q}$.
\end{lemma}

\begin{lemma}\label{min_reste}
Let $c$ and $r$ be real numbers such that $0<c<1$ and $0<r<1-c$. Assume that
$a_1,\ldots,a_n$ are complex numbers such that $0<|a_k|\ie 1$ and 
$\left|\frac{c - a_k}{1-c a_k}\right|\se r$ for all $k$, we have :
\[ 
\prod_{k=1}^n \left|\frac{c-a_k}{1-c a_k}\right| \se r^\beta 
\]
where $\beta =\log \left(\prod_{k=1}^n |a_k|\right) / \log\left(\frac{c+r}{1+c r}
\right)$.
\end{lemma}

\section{Upper Estimation of $|P(c)|$} \label{P_4}

\begin{theorem}[]\label{rapport}
For all 
$\delta\in]0,a[$, we have :
\[
\Big|\frac{P(\delta)}{P'(a)}\Big| \se \frac{1 - \sqrt{1 +\delta^2 - \delta a}}{n}
\]
\end{theorem}
\begin{proof}[Proof]
Write $R =\frac{nP(\delta)}{P'(a)}$ and let $z_c$ denote the complex number satisfying 
$|z_c-a|=|z_c|=1$ with positive
imaginary part. 
Theorem \ref{thm:3} asserts
that the disk of center $a$ and radius $|R|$ contains a complex 
number $\gamma$ such that $P(\delta)=P(\gamma)$, there is no loss of generality in 
assuming $\Im(\gamma)\se 0$. By theorem \ref{media}
the perpendicular bisector 
of the line segment $[\delta,\gamma]$ intersect the convex hull  of the zeros 
of $P'$, therefore $|z_c - \delta| \se |z_c - \gamma|$.
We deduce that :
\[
|R| \se 1 - |z_c - \gamma|  
\iff \left|\frac{P(\delta)}{P'(a)}\right| \se \frac{1 - \sqrt{1 +\delta^2 - 
\delta a}}{n}
\]
\vskip -10mm
\begin{figure}[h]
	\centering
	\includegraphics[trim=0 160 0 190,clip,scale=.4]{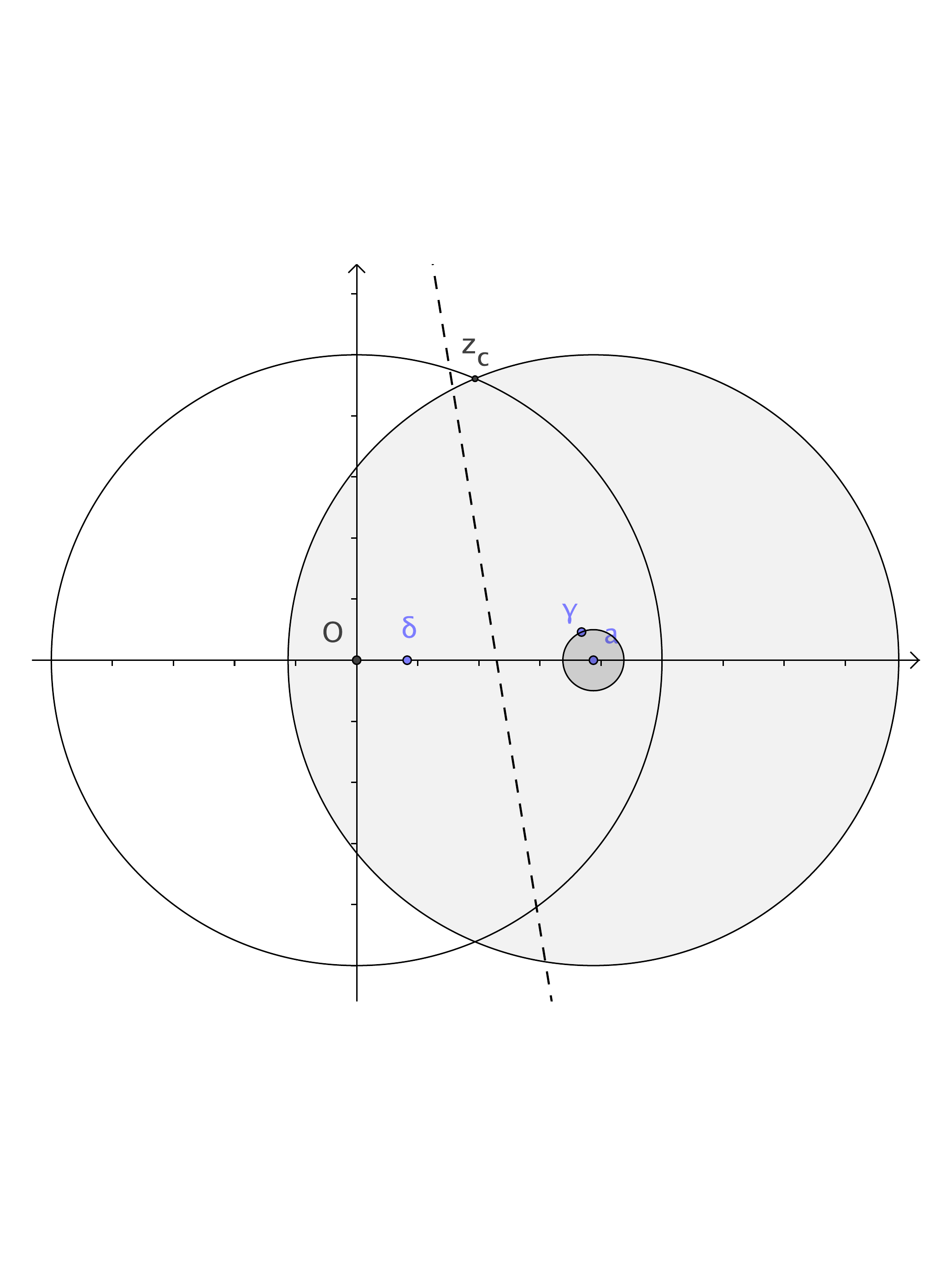}
	\caption{Diagram of the proof}
\end{figure}

\end{proof}

\begin{cor}\label{cor:m} We have :
\[
m \ie \inf_{\delta \in ]0,a]} \left( \frac{\delta}2 - \frac1{\delta n} \log(1 -\sqrt{1+
\delta^2 - \delta a}) \right)
\]
\end{cor}
\begin{proof}[Proof]
Let $\delta \in ]0,a]$, theorem \ref{rapport} gives :
\begin{eqnarray*}
|P(\delta)| & \se & \frac{1 - \sqrt{1 +\delta^2 - \delta a}}{n} |{P'(a)}| \\
& \se &  1 - \sqrt{1 +\delta^2 - \delta a}
\end{eqnarray*}
using lemma \ref{lem1} we have $|P(\delta)|^{1/n} \ie \sqrt{1+\delta^2-2\delta m}$
then :
\[
( 1 - \sqrt{1 +\delta^2 - \delta a} )^{1/n} \ie  \sqrt{1+\delta^2-2\delta m}
\]
we deduce that :
\begin{align*}
& \frac{1}{n} \log(1 - \sqrt{1 +\delta^2 - \delta a}) \ie \frac{\delta^2-2\delta m}{2}\\
\iff & 
m \ie \frac{\delta}2 - \frac1{\delta n} \log(1 -\sqrt{1+
\delta^2 - \delta a}) 
\end{align*}
which establishes the inequality of the lemma.
\end{proof}

\begin{remark}
When $n\to+\infty$,  $\inf_{\delta \in ]0,a]} \left( \frac{\delta}2 
	- \frac1{\delta n} \log(1 -\sqrt{1+\delta^2 - \delta a}) \right)\to0$.
\end{remark}
\begin{definition}
We define $N_1$ by the relation :
\[
n\se N_1 \iff 
\left(\frac{1+a/2}{1+a}\right)^{q(n-1)} \ie \frac{1 - \sqrt{1 -a^2/4}}{n a}
\]
\end{definition}

\begin{theorem}[]\label{thm_P4}
If $n\se N_1$, we have :	
\begin{equation}\label{majo:deri}  
|P'(a)|  \ie \frac{16 n}{a^2} \qquad\text{and}\qquad
|P(0)|  \se \frac{a^2}{16}
\end{equation}
\end{theorem}

\begin{proof}[Proof]
Assume that $n\se N_1$,
lemma \ref{maj_deri} asserts that for all $0<\delta<a/2$ :
\begin{align*}
\left|\frac{P'(\delta)}{P'(a)}\right| 
& \ie \left[ \left(\frac{1+\delta}{1+a}
\right)^{q} \sqrt{1 + \delta^2 - \delta a}^{1-q} \right]^{n-1}\\
& \ie \left(\frac{1+a/2}{1+a}\right)^{q(n-1)} \ie \frac{1 - \sqrt{1 -a^2/4}}{n a}
\end{align*}

We deduce that :
\begin{align*}
\left|\frac{P(0)}{P'(a)} - \frac{P(a/2)}{P'(a)}\right|  \ie &
\frac{a}{2} \sup_{\delta\in[0,a/2]}\left|\frac{P'(\delta)}{P'(a)}\right| \\
\ie & \frac{1 - \sqrt{1 -a^2/4}}{2 n}
\end{align*}
By theorem \ref{rapport}, we know that :
\[ 
\left|\frac{P(a/2)}{P'(a)}\right| \se \frac{1-\sqrt{1 - a^2/4}}{n} 
\]
then :
\[
\left|\frac{P(0)}{P'(a)}\right| \se \frac{1-\sqrt{1 - a^2/4}}{2n} \se \frac{a^2}{16 n}
\]
we deduce the theorem.
\end{proof}
\begin{remark}
This theorem expresses that the zeros of $P$ should lie nearby the unit circle
and those of $P'$ close to the circle $|z-a|=1$.
\end{remark}

From now on $c$ denote a real number such that $0<c<a$, we write :
\[
D = \max\left\{\left(\frac{1}{1+a}\right)^{q} \;;\;
\left(\frac{1+c}{1+a}\right)^{q}
\left(\sqrt{1+c^2-ac}\right)^{1-q} \right\} <1
\]
\begin{definition}
Define $N_2$ by the relation :
\[
n\se N_2 \iff 
D^{n-1} \ie \frac{a}{16 n}
\]
\end{definition}

\begin{theorem}\label{thm:mino} 
If $n\se \max\{N_1,N_2\}$ we have :
\[
|P(c)|\ie 1+a
\]
\end{theorem}
\begin{proof}[Proof]
Consider the mapping :
\[ \begin{array}{ccccl}
	f & : & [0, c] & \lra & \mathbb R \\
	& &  x & \longmapsto & \log\left[ \left(\frac{1+x}{1+a}\right)^{q} 
\sqrt{1 + x^2 - x a}^{1-q} \right]
\end{array}\] 
the first derivative of $f$ is given by :
\[ f'(x) = \frac{ x^2 + (1- a +m) x - m}{(1+x)(1+x^2 -x a)} \]
since $f'(x)$ shares the sign of : $x^2 + (1- a +m) x - m$, with lemma \ref{maj_deri},
we have :
\[
\left|\frac{P'(\delta)}{P'(a)}\right| \ie  D^{n-1} \qquad\text{then}\qquad
\sup_{\delta\in[0,c]} \left| \frac{P'(\delta)}{P'(a)}\right| \ie \frac a{16 n}
\]
For all $n\se \max\{N_1,N_2\}$, using theorem \ref{thm_P4} we deduce :
\[
|P(0) - P(c)| \ie \sup_{\delta\in[0,c]}\left|\frac{P'(\delta)}{P'(a)}\right|
\, c \, |P'(a)| 
\ie \frac{c}{a}  \ie 1
\]
then $|P(c)| \ie 1+a$. 
\end{proof}

\section{Lower Estimation of $|P(c)|$}
We prove now that there exist constants $C>0$ and $K>1$ such that if $n$ is large enough
the inequality $ |P(c)|\se C K^n$ holds.
Three new lemmas are needed.
\begin{lemma}\label{sym} 
Let $h$ and $c$ be positive real numbers satisfying $0<c<1-h$. For all $z\in\mathbb C$ 
satisfying $|z|\se 1 - h$, we have :
\[ 
|c-z| \se \frac{c}{1-h} \left|\frac{(1-h)^2}{c} - z\right|  
 \]
\end{lemma}

\begin{lemma}\label{mini} 
Let $b>1$, we have :
\[ 
(b-a)\prod_{k=1}^{n-1} |b - z_k| \se (b-1) \prod_{k=1}^{n-1} |b - \zeta_k|
 \]
\end{lemma}

\begin{lemma}\label{exclu2} 
Assume that  $0<h<c<a<1-h$ and
define the disk $\mathcal D$ by :
\[ 
\mathcal  D := \left\{ z \in \mathbb C \;;\;
\left|\frac{(c-z)}{(1-h)^2 - c z}\right| \ie \frac{c(a-c)}{2((1-h)^2 - c^2)} 
\right\}
\]
Then $\mathcal D$ contains no zero of $P$. 
\end{lemma}
\begin{definition}
Let $\alpha$, $r$ and $K$ be defined by $r=\frac{c(a-c)}{2(1-c^2)}$, 
$\alpha = \log\left(\frac{a}{16}\right) / \log\left(\frac{c+r}{1+cr}\right)$, and
\[
K = \min\left\{ (1+c-ac)^p\sqrt{1+c^2-ac}^{1-p} \;;\; (1+c)^q \sqrt{1+c^2-ac}^{1-q}
\right\}
\]
\end{definition}
\begin{remark}
Observe that if $c$ is sufficiently close to $a$ then $K>1$.
\end{remark}

\begin{theorem}[]\label{thm:majo}
If $n\se N_1$, we have :
\[|P(c)| \se \frac{(1-c)(a-c)}{1-ac}\, r^\alpha \, K^{n-1}
\]
\end{theorem}

\begin{proof}[Proof]
Fix $0 < h <1-a$ and assume that the zeros of $P$ are indexed such that for all 
$k\se n_0$ we have $|z_k|\se 1-h$.
Let $b_h = \frac{(1-h)^2}{c}$, we have : 
\begin{align*}
\left|P(c)\right| = & |c-a| \prod_{k=1}^{n_0-1} |c- z_k| \; 
\prod_{k=n_0}^{n-1} |c- z_k| \\
\se & (a-c) \prod_{k=1}^{n_0-1} |c- z_k| \; \left(\frac{c}{1-h}\right)^{n-n_0}
\; \prod_{k=n_0}^{n-1} |b_h - z_k| \qquad\qquad\text{(using lemma \ref{sym})}\\
= & \left(\frac{a-c}{b_h-a}\right) \prod_{k=1}^{n_0-1} \left|\frac{c- z_k}{b_h - z_k}\right| \; 
\left(\frac{c}{1-h}\right)^{n-n_0} \; |b_h-a| \prod_{k=1}^{n-1} |b_h - z_k| 
\end{align*}
lemma \ref{mini} gives :
\begin{align*}
|P(c)| \se & \left(\frac{a-c}{b_h-a}\right) \prod_{k=1}^{n_0-1}\left|\frac{c- z_k}
{b_h - z_k}\right| \; 
\left(\frac{c}{1-h}\right)^{n-n_0} (b_h-1) \prod_{k=1}^{n-1} |b_h - \zeta_k| \\
\se & \left(\frac{(b_h-1)(a-c)}{(b_h-a)}\right) 
\prod_{k=1}^{n_0-1} \left|\frac{(c- z_k)(1-h)}{(b_h - z_k)c}\right| \; 
\left(\frac{c}{1-h}\right)^{n-1} \prod_{k=1}^{n-1} |b_h - \zeta_k| 
\end{align*}
Using lemma \ref{min_deri}, we obtain :
\begin{equation}\label{eq:2}
|P(c)| \se \left(\frac{(b_h-1)(a-c)}{(b_h-a)}\right) 
\prod_{k=1}^{n_0-1} \left|\frac{(c- z_k)(1-h)}{(b_h - z_k)c}\right| \; (K_h)^{n-1}
\end{equation}
where
\[
K_h = \frac{c}{1-h}\min\left\{ (1+b_h-a)^p \sqrt{1+b_h^2-ab_h}^{1-p} \;;\; 
(1+b_h)^q \sqrt{1+b_h^2-ab_h}^{1-q} \right\}
\]
Let $c'=\frac{c}{1-h}$, $z_k'=\frac{z_k}{1-h}$ and 
$r_h =\frac{c(a-c)(1-h)}{2((1-h)^2-c^2)}$, by lemma \ref{exclu2} we know that :
\[ 
\left|\frac{(c- z_k)(1-h)}{(b_h - z_k)c}\right|
= \left|\frac{c'- z_k'}{1 - c' z_k'}\right|  \se r_h
\]
Lemma \ref{min_reste} gives that :
\begin{equation}\label{eq:3}
\prod_{k=1}^{n_0-1} \left|\frac{c'- z_k'}{1 - c' z_k'}\right| \se r_h^{\beta_h}
\end{equation}
where $\beta_h=\log\left(\prod_{k=1}^{n_0-1} |z_k'|\right) /
\log\left(\frac{c'+r_h}{1+c' r_h}\right)$, by theorem \ref{thm_P4} we have that :
\[ 
\prod_{k=1}^{n_0-1} |z_k'| \se \prod_{k=1}^{n_0-1} |z_k| \se \prod_{k=1}^{n-1} |z_k| 
\se \frac{a}{16}
\]
therefore in \eqref{eq:3} the constant $\beta_h$ can be replace by  
$\alpha_h=\log\left(\frac{a}{16}\right) /
\log\left(\frac{c'+r_h}{1+c' r_h}\right)
$. Combining \eqref{eq:2} with \eqref{eq:3} we deduce that :
\[
|P(c)| \se \left(\frac{(b_h-1)(a-c)}{(b_h-a)}\right) r_h^{\alpha_h} (K_h)^{n-1}
\]
the theorem follows letting $h\to 0$.
\end{proof}

\section{Main Result and Conclusion}

We can now formulate our main result. Assume that $K>1$ and let  
\[ N = \max\left\{ N_1 ; N_2 ;
\frac{\log\left(\frac{(1+a)(1-a c)}{(1-c)(a-c)}\right) - \alpha
\log(r)}{\log(K)} + 1 \right\}
\]
\begin{theorem}[]\label{conclu}
Sendov conjecture holds for all polynomial $P$ satisfying
$\deg(P) \se  N$.
\end{theorem}
\begin{proof}[Proof]
Combining the results of the theorems
\ref{thm:mino} and \ref{thm:majo} we obtain :
\begin{align*}
& 1+a \se \frac{(1-c)(a-c)}{1-ac}\, r^\alpha \, K^{n-1}\\
\iff & \log\left(\frac{(1+a)(1-a c)}{(1-c)(a-c)}\right) \se \alpha
\log(r) + (n-1) \log K \\
\iff & n \ie \left[\log\left(\frac{(1+a)(1-a c)}{(1-c)(a-c)}\right) - \alpha
\log(r)\right]/ \log(K) + 1
\end{align*}
this gives the theorem.
\end{proof}
To compute $N$
we need the value of $m$ which is unknown but can be replace by the upper estimate
given by corollary \ref{cor:m} which depends only on $n$. Computations can be done 
in the following way :
\begin{enumerate}
	\item  choose arbitrarily $0<c<a$ and $m>0$ ;
	
	\item  compute $K$ and check that $K>1$ if not go back to first step modifying 
	$c$ ; 
	
	\item compute $N$ and deduce the upper estimate of $m$ given by 
	corollary \ref{cor:m}, if it is bigger than $m$ go back to first step increasing $m$ 
	else decreasing $m$ repeat until equality holds ;
	
	\item  adjust the choice of $c$ to obtain the smallest value of $N$.
\end{enumerate}
Finally, find below the values of $N$ obtained for many choices of $a$. 
\[
\begin{array}{|c|c|c|c|c|c|c|c|c|}
\hline
a    & c & m & r & \alpha & p & q & K & N  \\ \hline
0,9	& 0,756 & 0,080 &	0,1270 & 13,32 & 0,673 &	0,255  & 1,031 & 	1006 \\ \hline
0,8	 & 0,700 & 0,100 & 0,0686 & 9,66	 & 0,500 & 0,214 & 1,049 & 616 \\ \hline
0,7 & 0,630 & 	0,110 & 0,0366 & 7,3	& 0,369 & 	0,178 & 1,051 & 560 \\ \hline
0,6	& 0,550	 & 0,100 & 0,0197 & 5,73 & 0,286 & 0,154 & 	1,048 & 563 \\ \hline
0,5 & 0,460 &	0,100 & 0,0117 & 4,58 & 0,200 &	0,120 & 1,035 & 718 \\ \hline
0,4 & 0,374 & 	0,089 & 0,0057 & 3,8 & 0,139	& 0,093 &	1,024 & 1004 \\ \hline
0,3 & 0,284 & 	0,073 & 0,0025 & 3,18 & 0,091 & 	0,067 & 1,014 & 1654 \\ \hline
0,2 & 0,191 & 	0,053 & 0,0009 & 2,65 & 0,052 &	0,043 & 1,007 & 3587 \\ \hline
0,1 & 0,096 &	0,029 & 0,0002 & 2,17 & 0,022 & 0,020 & 1,002 & 15064 \\ \hline
\end{array}
\]
\vskip 5mm
{\bf Conclusion.} 
It may be surprising to see that Sendov conjecture is easily proved in extremal
cases, e.g. when $a=0$ or $a=1$ and that in the generic cases i.e. $0<a<1$ only very
partial results are known about it. In the present paper I want to fill this
lack but it remains work to obtain a definitive proof of the conjecture i.e. 
to prove $N=8$ for all $a\in]0,1[$.

\section{Proofs of lemmas}

\begin{proof}[Proof of lemma \ref{lem1}]
Let $k\in\{1,...,n\}$.
Fix $\Re(a_k)$ then $|\delta - a_k|$ is maximum when $|a_k|=1$. Therefore
we can assume that $|a_k|=1$ and write $a_k=e^{i\theta_k}$.
The mapping $\Phi(x)= \frac 1 2 \log(1+\delta^2-2\delta x)$ is concave, by
Jensen's inequality we get :
\begin{align*}
	\log\left(\left|\prod_{k=1}^n (\delta - a_k)\right|^{1/n}\right) & =  \frac{1}{n} 
	\sum_{k=1}^n \log|\delta - a_k| \\
	& =  \frac 1 n
	\sum_{k=1}^n \Phi(\cos\theta_k) \\
	& \ie  \Phi\left(\frac 1 n  \sum_{k=1}^n \cos\theta_k\right)\\
	& =  \log\left(\sqrt{1+\delta^2 - 2 \delta s}\right) 
\end{align*}
which establishes the lemma.
\end{proof}

\begin{proof}[Proof of lemma \ref{maj_deri}]
	Let $k\in\{1,\ldots,n-1\}$, for $\Re(b_k)$ given the modulus
	$\left|\frac{\delta - b_k}{a - b_k}\right|$ is maximal if 
	$|b_k| = 1$, therefore we can assume that for all $k$, $|b_k|=1$.
	The mapping $\Phi$ defined on $[-1, a/2]$ by
	\[ 
	\Phi(x)= \frac 1  2 \log\left(\frac{1+\delta^2 - 2 \delta x}
	{1 + a^2 - 2 a x} \right)
	\] 
	is convex, since :
	\[
	\Phi''(x)  = \frac{2 (a-\delta) (1-a\delta) ((a+\delta)(1+a\delta) - 2 a \delta x)}
	 {(1+\delta^2 - 2 \delta x)^2 (1+a^2 - 2 a x)^2} \se 0
	 \]
	 We deduce that :
	 \begin{align*}
	 \log\left(\prod_{k=1}^{n-1} \left| \frac{\delta - b_k}{a - b_k}\right|
	 \right)  = & \sum_{k=1}^{n-1} \Phi(\mathrm{Re}(b_k)) \\
	  \ie & (n-1) \left[q \Phi(-1) + (1-q) \Phi\left(\frac a 2\right)\right]\\
	  \ie & (n-1) \left[ q \log\left(\frac{1+\delta}{1+a}
	 \right) + (1-q) \log\sqrt{1 + \delta^2 - \delta a} \right]
	 \end{align*}
	 where $q$ satisfies : $-q + (1-q) \frac a 2 = s$, therefore $q=\frac{a/2-s}
	 {1+a/2}$. 
	 Applying the exponential function to both sides we obtain the lemma.
\end{proof}

\begin{proof}[Proof of lemma \ref{min_deri}]
Let $k\in\{1,\ldots,n-1\}$, for $\Re(b_k)$ given the modulus
$\left|b - b_k\right|$ is minimal if $|a-b_k| = 1$ or if $b_k\in [-1,-1+a]$. 
Assume that $|a-b_k| = 1$ or 
$b_k\in[-1,a-1]$ for every $k$. The mapping :
\[ \begin{array}{ccccl}
	\Phi & : & [-1, a/2] & \lra & \mathbb R \\
	& &  x & \longmapsto & \left\{ \begin{array}{ccc}
	\frac 1 2\log\left(1+(b-a)(a+b- 2x)\right)
	& \text{if} & x \se a-1 \\
	\log(b-x) & \text{if} & x \in[-1,a-1]
	\end{array}\right.
\end{array}\] 
satisfies $\Phi''(x)\ie 0$ then $\Phi$ is concave on $[-1,-1+a]$ and 
$[-1+a,a/2]$. We deduce that :
\begin{align*}
\log\left(\prod_{k=1}^{n-1}|b - b_k|\right)  & = \sum_{k=1}^{n-1} \Phi(\Re(b_k))\\
& \se (n-1) \min\{ \alpha \Phi(-1) + \beta\Phi(-1+a) +\gamma \Phi(a/2) \}
\end{align*}
where the minimum is taken over the set $(\alpha,\beta,\gamma)\in
\mathbb R_+^3$ such that :
\[\left\{\begin{array}{ccc}
\alpha+\beta+\gamma & = & 1 \\
-\alpha + (a-1)\beta + \frac a 2\gamma & = & s
\end{array}\right.\]
Let us define the mappings $g$, $f_1$, $f_2$ : $\mathbb R_+^3 \longrightarrow \mathbb R$ by	
\begin{align*}
	 g(\alpha,\beta,\gamma) & = \alpha \log(b+1) + \beta \log(b+1-a) +
	\frac \gamma 2 \log(1+b^2-ab)
	\\
	f_1(\alpha,\beta,\gamma) & =  \alpha+\beta+\gamma
	\\
	f_2(\alpha,\beta,\gamma) & =  -\alpha + (-1+a)\beta + \frac a 2 \gamma
\end{align*}
We want to compute :
\[
\min_{(\alpha,\beta,\gamma)\in\mathbb R_+^3}\left\{g(\alpha,\beta,\gamma) \;;\; 
f_1(\alpha,\beta,\gamma)=1 \text{ and } f_2(\alpha,\beta,\gamma)=m \right\} 
\]
The Lagrange multipliers theory asserts that if the minimum is reached at 
$(\alpha_0,\beta_0,\gamma_0)\in(\mathbb R_+^*)^3$ then there exist multipliers
$\lambda_1$ and $\lambda_2$ such that :
\[ \nabla g = \lambda_1 \nabla f_1 + \lambda_2 \nabla f_2 \iff
\left(\begin{array}{c} \log(b+1) \\ \log(b+1-a) \\ \frac 1 2 \log(1+b^2-ab)
\end{array}\right) = \lambda_1 \left(\begin{array}{c} 1 \\ 1 \\ 1 \end{array}\right)
+ \lambda_2 \left(\begin{array}{c} -1 \\ -1 + a \\ a/2 \end{array}\right)
\]
Which is impossible in generic case, we deduce (if necessary slightly modifying 
$a$ or $b$) that $\alpha=0$ or $\beta=0$. Let us consider both cases.
\begin{itemize}
	\item If $\alpha=0$, we have :
	\[ 
	\prod_{k=1}^{n-1}|b - b_k| \se \left((b+1-a)^\beta (\sqrt{1+b^2-ab})^{1-\beta}
	\right)^{n-1}
	\]
	where $ \beta = \frac{a/2 - s}{1 - a/2}$
	then $\prod_{k=1}^{n-1} \left| b - b_k\right| \se B_1^{n-1} $.
	 
	 \item If $\beta=0$, we have :
	 \[ \prod_{k=1}^{n-1}|b - b_k| \se \left((b+1)^\alpha (\sqrt{1+b^2-ab})^{1-\alpha}
	 \right)^{n-1}\]
	 where  $\alpha = \frac{a/2 - s}{1 + a/2}$
	then $\prod_{k=1}^{n-1} \left| b - b_k\right| \se B_2^{n-1}$.
\end{itemize}
This completes the proof of lemma \ref{min_deri}.
\end{proof}

\begin{proof}[Proof of lemma \ref{min_reste}]
Let $k\in\{1,\ldots,n\}$, if $\left|\frac{c - z_k}{1 - c z_k}\right|$ is given, $|z_k|$
is maximal when $z_k$ is a real number with $\frac{c+r}{1+c r} \ie z_k \ie 1$. 
For all $k\in\{1,\ldots,n\}$, assume that $z_k\in [\frac{c+r}{1+c r},1]$,
write $\alpha_k =\log(z_k)$ and consider the mapping $\Phi$ defined on 
$\left[\log(\frac{c+r}{1+c r}) , 0 \right]$ by
\[ 
\Phi(\alpha) = \log\left(\frac{e^\alpha - c}{1 - c e^\alpha}\right) 
\]
It's easily seen that $\Phi$ is concave. 
Therefore :
\begin{align*}
	 \log\left(\prod_{k=1}^n \left|\frac{c-z_k}{1-c z_k}\right|\right) & =  
	 \sum_{k=1}^N \Phi(\alpha_k) \\
	 & \se \beta \Phi\left(\log(\frac{c+r}{1+c r})\right) + (N-\beta) \Phi(0)
\end{align*}
where $\beta \log(\frac{c+r}{1+c r}) = \log\left(\prod_{k=1}^n|z_k|\right)$,
lemma \ref{min_reste} is obtained taking exponential.
\end{proof}

\begin{proof}[Proof of lemma \ref{sym}]
The following inequalities are equivalent :
\begin{align*}
& |c-z| \se \frac{c}{1-h} \left|\frac{(1-h)^2}{c} - z\right|  \\
\iff & \left|\frac z{1-h} - \frac c{1-h}\right|^2 \se \left| 1 - \frac c{1-h}
\frac z{1-h} \right|^2 \\
\iff & \left|\frac z{1-h}\right|^2 + \left(\frac c{1-h}\right)^2 \se 1 + 
\left|\frac{c z}{(1-h)^2}\right|^2 \\
\iff &  \left(\left|\frac z{1-h}\right|^2 - 1\right) 
\left(1 - \left(\frac c{1-h}\right)^2 \right) \se 0
\end{align*}
we deduce the lemma.
\end{proof}

\begin{proof}[Proof of lemma \ref{mini}]
Let us compute the quotient :
\begin{align*}
\left|\frac{P'(b)}{P(b)}\right| = & \left|\frac1{b-a} + \sum_{k=1}^{n-1}
\frac1{b-z_k} \right| \\
\ie & \frac1{b-a} + \sum_{k=1}^{n-1} \frac1{|b-z_k|} 
\ie \frac n{b-1} 
\end{align*}
then $|P(b)| \se \frac{b-1}{n} |P'(b)| = (b-1)\prod_{k=1}^{n-1}|b-\zeta_k|$ 
and the lemma follows.
\end{proof}

\begin{proof}[Proof of lemma \ref{exclu2}]
Define $k=\frac{c(a-c)}{2((1-h)^2 - c^2)}$,
the disk $\mathcal D$ has center $\omega$ and radius $R$ given by :
\[ 
\omega = c\frac{1 - k^2(1-h)^2}{1-(kc)^2} \qquad \text{and}\qquad
R = k\frac{(1-h)^2-c^2}{1-(kc)^2} 
\]
according to theorem \ref{exclu} it suffices to show that  :
\[ 
R \ie 1 - \sqrt{1+\omega^2 - \omega a} 
\]
or equivalently
\[
\omega^2 - R^2 \ie \omega a - 2 R \qquad\text{and}\qquad R \ie 1
\]
The first inequality holds since :
\[
\omega^2 - R^2 \ie \omega a - 2 R \qquad
\iff \qquad k^2 (1-h)^2 ((1-h)^2-ac)  \se 0
\]
On the other hand :
\[
R \ie 1 \qquad
\iff \qquad \frac{c(a-c)}{2} \ie  1 - k^2 c^2
\]
therefore the second inequality is straightforward since $k\ie\frac12$ and this proves
the lemma.
\end{proof}

\bibliographystyle{amsalpha}

\end{document}